\documentclass[12pt]{amsart}
\usepackage{amsmath, amssymb, amscd, amsthm, amsfonts}
\usepackage{graphicx}
\usepackage{hyperref}
\usepackage{inputenc}[utf8]
\usepackage{xcolor} %for color in the text

\usepackage[pagewise]{lineno}%\linenumbers %to number the lines

\usepackage{pinlabel} %For putting labels on figures

\title[Twice-punctured tori in a knot exterior]{A bound on the number of twice-punctured tori in a knot exterior}
\author[R. Aranda, E. Ram\'irez-Losada \and J. Rodr\'iguez-Viorato]{Román Aranda \and Enrique Ram\'irez-Losada \and Jes\'us Rodr\'iguez-Viorato}
\date{}

\newtheorem{theorem}{Theorem}
\newtheorem{definition}[theorem]{Definition}
\newtheorem{lemma}[theorem]{Lemma}
\newtheorem{conjecture}[theorem]{Conjecture}
\newtheorem{proposition}[theorem]{Proposition}

\begin{document}

\maketitle

\begin{abstract}
This paper continues a program due to Motegi regarding universal bounds for the number of non-isotopic essential $n$-punctured tori in the complement of a hyperbolic knot in $S^3$. For $n=1$, Valdez-S\'anchez showed that there are at most five non-isotopic Seifert tori in the exterior of a hyperbolic knot. In this paper, we address the case $n=2$. We show that there are at most six non-isotopic, nested, essential 2-holed tori in the complement of every hyperbolic knot. 
\end{abstract}

\section{Introduction}
It is a celebrated result of W. Jaco, P. Shalen, and K. Johannson that every Haken 3-manifold admits a unique decomposition along essential tori \cite{JSJ_Jaco_Shalen,JSJ_Johannson}. The resulting pieces are Seifert fibered spaces and other simple 3-manifolds. %These are called JSJ pieces. 
K. Motegi asked if the number of JSJ pieces in a 3-manifold obtained by Dehn surgery on a hyperbolic knot in $S^3$ is bounded. With this in mind, we are interested in finding a universal bound in the number of pairwise disjoint, essential punctured tori in the complement of hyperbolic knots. In \cite{tsutsumi}, Y. Tsutsumi gave a partial answer to Motegi's question by showing that every hyperbolic knot admits at most seven non-isotopic and disjoint genus one Seifert surfaces in its complement. In \cite{luis1}, L. Valdez-S\'anchez found the efficient lower bound by showing that there are at most five such once-punctured tori. 
In this paper, we give a first universal bound for the number of twice-punctured torus in the complement of a hyperbolic knot. 
\begin{theorem}\label{thm:cota-2toro-anidado}
Let $K$ be a hyperbolic knot in $S^3$. There are at most six pairwise disjoint, non-isotopic, nested, embedded twice-punctured tori with an integral slope in the complement of $K$. 
\end{theorem}
We believe that our upper bound is not efficient. Recently, M. Eudave-Mu\~noz and M. Teragaito built infinite families of hyperbolic knots with four non-isotopic twice-punctured tori in their exterior (personal communication). 
\begin{conjecture}
There are at most four pairwise disjoint, non-isotopic, nested, embedded twice-punctured tori with an integral slope in the complement of any hyperbolic knot in $S^3$.
\end{conjecture}
Our proofs build on the work of Tsutsumi and Valdez-S\'anchez studying incompressible surfaces in genus two handlebodies. Along the way, we discover some technical lemmas that may be of interest to those solving tangle equations (See Propositions \ref{ref_prop2}, \ref{prop_tangle_argument} and \ref{prop_new_lem_3.5}). Expanding on the fact that genus two surfaces in the complement of hyperbolic knots in $S^3$ bound handlebodies on one of their sides (see \cite[Section 3]{tsutsumi} for details), Theorem \ref{thm:cota-2toro-anidado} is a direct corollary of the following result.
\begin{theorem}\label{thm:cota-2toro-en-dos-asas}
Let V be a handlebody of genus two, and let $J = J_1 \cup J_2$ be an essential pair of simple closed curves that separates an annulus from $\partial V$.
Then, $J$ bounds at most three mutually disjoint, non-parallel, incompressible, separating, twice-punctured tori in $V$.
\end{theorem}

%\textbf{Outline of the work.} This paper builds-up to the proof of Theorem \ref{thm:cota-2toro-en-dos-asas} in Section \ref{section:cota-dos-asas}. 

\section{Preliminaries}
In this section, we review some results from \cite{luis1} and \cite{tsutsumi} that we will use later in this document. We begin introducing some basic vocabulary in the following definition.

\begin{definition}
Let $\gamma$ and $\gamma'$ be simple closed curves in $\partial H$.
\begin{itemize}
    \item A \emph{Companion annulus} for $\gamma$ is an annulus $A$ properly embedded in $H$ such that $\partial A$ bounds an annulus $B$ in $\partial H$ with $\gamma$ in the core of $B$.
    \item We say that $\gamma$ is a \emph{power circle} if  $[\gamma] = \alpha^n$ for some $\alpha \in \pi_1(H)$ and $n \geq 2$.
    \item The curve $\gamma$ is called a \emph{primitive curve} if $[\gamma]$ is part of a base for $\pi_1(H)$.
    \item We say that $\{\gamma, \gamma'\}$ are basic circles if they represent a basis for $\pi_1(H,*)$ for some base point. 
\end{itemize}
\end{definition}

The following results come from \cite[Prop 3.3]{luis1}; we extract the parts we need here.

\begin{proposition}[L. Valdez, \cite{luis1}]\label{pro:gamma-is-not-primitive-or-power}
Let $H$ be a genus two handlebody and $\gamma\subset \partial H$ a curve non-trivial in $H$. The surface $\partial H - \gamma$ compresses in $H$ iff $\gamma$ is a primitive or a power circle in $H$
\end{proposition}

\begin{proposition}[L. Valdez, \cite{luis1}] \label{pro:gamma-has-companion-annulus-condition}
Let $H$ be a genus two handlebody and $\gamma\subset \partial H$ a curve non-trivial in $H$. Then $\gamma$ has a companion annulus in H iff $\gamma$ is a power circle in $H$. In this case, there is a unique non-separating compressing disk $D$ for $H$ with $\partial D$ disjoint from $\gamma$. 
\end{proposition}

\begin{proposition}[L. Valdez, Lem 3.5(2) \cite{luis1}]\label{pro:lem_3.5_part2}
Let $H$ be a genus two handlebody and $\gamma, \gamma'\subset \partial H$ a pair of disjoint circles. Let $M=V\cup_\gamma H \cup_{\gamma'}V'$ be a manifold obtained by gluing solid tori $V$, $V'$ to $H$ along disjoint annular neighborhoods $A=\partial H \cap \partial V$ and $A'=\partial H\cap \partial V'$ of $\gamma$ and $\gamma'$, respectively, where each annulus $A$, $A'$ runs at least twice around $V$, $V'$, respectively. Then $M$ is a genus two handlebody if and only if $\gamma$ and $\gamma'$ are basic circles in $H$. 
\end{proposition}

%%%%%%%%%%%%%%%%%%%%%%%%%%%%%%%%%%%%%%%%%%
%%%%%%%%%%%%%%%%%%%%%%%%%%%%%%%%%%%%%%%%%%%%%%

\subsection{Simple pairs}
This section will discuss some results about pairs $(H,J)$ where $H$ is a genus two handlebody and $J \subset \partial H$ is an essential separating circle. We will say that a pair $(H, J)$ is
\begin{enumerate}
    \item \emph{trivial} if it is homeomorphic to the pair $(T \times I, \partial T \times \{0\})$ 
    \item \emph{minimal} if any once-punctured torus $T$ in $H$ with $\partial T = J$ is parallel to $\partial H$ relative to $J$; in particular any trivial pair is minimal. 
\end{enumerate}
% FIXME: Tal vez deberiamos referir a que esta notacion es de Luis Valdez

Now we define a special family of pairs $(H,J)$ introduced in  \cite{tsutsumi}. 

\begin{definition}
A pair of type $(p_0,p_1)$ is a tuple $(H,\gamma)$ constructed as follows:
%result of gluing two solid tori $V_0$ and $V_1$ with $T\times [0,1]$ ($T$ is a once-punctured torus) in the following way:
\begin{itemize}
    \item Take $T$ a once-punctured torus and $(a_0, a_1)$ a pair of simple closed curves that form a basis for $\pi_1(T)$.
    \item Take $(V_0, c_0)$ and $(V_1, c_1)$ two framded solid tori.  Where $c_i$ is a curve on $\partial V_i$ that goes $p_i$ times around the core of $V_i$ for $i=0,1$.
\end{itemize}

We construct $H$ as the space resulting from the identification of a regular neighborhood of $c_i \subset \partial V_i$ with one of $a_i \times \{i\} \subset \partial(T\times [0,1])$ for $i=0,1$. We take $\gamma = \partial T \times \{1/2\} \subset H$. In  \cite{tsutsumi} is shown that $H$ is a 2-handle body.

%FIXME. Need a Theorem.

%Let $T$ be a once-punctured torus, denote with $\gamma$ its boundary, and take a non-separating simple closed curve $c\subset T$. Let $V$ be a solid torus and $\alpha\subset \partial V$ be a $(a,p)$-closed curve; that is, $\alpha$ runs $p$ times around $V$ and $a$ times around the meridian of $V$. 

\end{definition}
%We begin describing a construction, due to Tsutsumi \cite{tsutsumi}, of pairs $(H,J)$ where $J$ bounds two or three non-isotopic incompressible tori. 

In the above definition, observe that a pair of type $(1,1)$ is a trivial pair, and a type $(1,p)$ pair is homeomorphic to the gluing of just one solid torus along a primitive curve on $T \times \{1\}$.

%FIXME: Agregar imagen, es sugerencia
The pairs $(1,p)$ appear naturally as part of pairs $(H,J)$ with multiple once-punctured tori as the following proposition states.

\begin{proposition}[L. Valdez, Rem 3.8 and Cor 3.10 of \cite{luis1}]\label{pro:lem_3.10}
Let $(H,J)$ be a pair with $H$ a genus two handlebody and $J\subset \partial H$ a curve separating $\partial H$ into two tori $T_1$, $T_2$. Suppose that $T'_1,T'_2$ are incompressible tori with boundary in $J$ dividing $H$ into non-product regions as follows $$H=H_1\cup_{T'_1} \cup H_0\cup_{T'_2} H_2.$$ Then $(H_1,J)$ and $(H_2,J)$ are pairs of type $(1,p_1)$ and $(1,p_2)$ for some $p_1,p_2>1$. 
\end{proposition}

%We finish this section with the following technical Lemma about pairs of type $(1,p)$.

\begin{lemma}\label{lem:simple_pair_annulus}
Let $(H,\gamma)$ be a pair of type $(1,p)$ with $p>1$. Let $T_0$ and $T_1$ be the once-punctured tori in $\overline{\partial H-\gamma}$ and let $\rho_i\subset T_i$ be a non-separating simple closed curve. If $\rho_0$ and $\rho_1$ cobound an embedded annulus in $H$, then $\rho_i$ is a power circle in $H$. 
\end{lemma}
\begin{proof}
%{\color{red} 
%TODO - 6 de septiembre pagina 17}
We will show that if $\rho_1$ and $\rho_0$ are freely homotopic in $H$, then $\rho_1$ is a power circle in $H$. 

Using the notation above, identify $T_0$ with $T\times \{0\}$ and $T_1$ with $\left(T\times \{1\}-\eta(c\times \{1\}) \right)\cup \left(\partial V -\eta(\alpha)\right)$. 
Let $*\in \partial H$ be a point in a boundary component of $\eta(c\times \{1\})=\eta(\alpha)$. We think of $c_1=c\times \{1\}$ and $\alpha$ as loops based at $*$. %Since $*\in T\times \{1\}$, we identify $*$ with a point in $T$ which we also denote by $*$. 
Let $b\subset T$ be a simple closed curve dual to $c$. We can pick $b$ so that $(b\times \{1\})\cap (c\times \{1\})=\{*\}$ and $b_1=b\times \{1\}$ is based at $*$. This way, $\pi_1(T\times [0,1],*)=\pi_1(T_0,*)=\langle b_1, c_1\rangle$. % is free of rank two. 
On the other hand, let $\delta\subset \partial V$ be a simple closed curve intersecting $\alpha$ at $*$. Let $b_2$ be the loop in $T_1$ based at $*$ obtained by following $b_1$ until reaching the other boundary of $\eta(c_1)$ and then following $\delta$ until reaching $*$. By construction, $b_2$ is dual to $\alpha$ in $T_1$ and so $\pi_1(T_1,*)=\langle \alpha, b_2\rangle$. 

Let $\lambda\subset V$ be a longitude of $V$ based at $*$. With this notation, $\delta = \lambda^r$, $\alpha = \lambda^p$ for some $r\neq 0$. In particular, $r$ and $p$ are relatively prime. Denote the embeddings of $F_i$ into $H$ by $\iota_i$. By construction, the following equations hold in $\pi_1(H,*)$: $\iota_0^*(c_1)=\lambda^p$, $\iota_0^*(b_1)=b_1$, $\iota_1^*(\alpha)=\lambda^p$, and $\iota_1^*(b_2)=b_1\lambda^r$. 
Via Seifert-Van Kampen Theorem, we obtain that
$\pi_1(H,*)=%\langle c_1,b_1,\alpha, \delta \mid c_1=\alpha \rangle = 
\langle b_1, \lambda\rangle$. 

We are ready to discuss the proof of the Lemma. Pick cyclically reduced words representing $\rho_i\in \pi_1(T_i,*)$ as follows 
$\rho_0=c_1^{n_1} b_1^{m_1}c_1^{n_2} b_1^{m_2}\cdots c_1^{n_k} b_1^{m_k}$, $\rho_1 = \alpha^{x_1}{b_2}^{y_1} \alpha^{x_2}{b_2}^{y_2}\cdots \alpha^{x_l}{b_2}^{y_l}$. 
The fact that $\rho_i$ is a non-separating simple closed curve implies that $\rho_i$ is primitive in $\pi_1(T_i,*)$. By \cite{primitive_words, palindromic_words}, $m_im_j>0$ and $n_in_j>0$ for all $i,j$ (resp. $x_ix_j>0$ and $y_iy_j>0$). %This comes from https://demonstrations.wolfram.com/PrimitiveElementsInTheFreeGroupOfRankTwo/
If $\rho_0$ and $\rho_1$ are freely homotopic in $H$, then $\iota_0^*(\rho_0)=\omega \iota_1^*(\rho_1)\omega^{-1}$ for some word $\omega$ in $\{b_1,\lambda\}$. In particular, 
\[ 
\lambda^{pn_1} b_1^{m_1}\lambda^{pn_2} b_1^{m_2}\cdots \lambda^{pn_k} b_1^{m_k}
=
\omega
\lambda^{px_1}(b_1\lambda^r)^{y_1} \lambda^{px_2}(b_1\lambda^r)^{y_2}\cdots \lambda^{px_l}(b_1\lambda^r)^{y_l}
\omega^{-1}. 
\]
Recall that the words from $\rho_0$ and $\rho_1$ are cyclically reduced. Suppose that $y_j>0$. From the equation above, notice that if some $y_j>1$, then there will exist $i$ satisfying $\lambda^r=\lambda^{pn_i}$, contradicting the fact that $\gcd(p,r)=1$. Then, $y_j=1$ for all $j$. In this case, for some $i,j$, $\lambda^{r+px_j}=\lambda^{pn_i}$ which contradicts the fact that $\gcd(p,r)=1$, $p>1$ and $r\neq 0$. In conclusion, $y_j=0$ for all $j$ and so $\rho_1=\lambda^{px}$ is a power %(or primitive) 
curve in $H$. 
\end{proof}

%%%%%%%%%%%%%%%%%%%%%%%%%%%%%%%%%%%%%%%%%%%%%%%%%%%%%%%%%%%%%%%%%%

%%%%%%%%%%%%%%%%%%%%%%%%%%%%%%%%%%%%%%%%%%%%%%
%%%%%%%%%%%%%%%%%%%%%%%%%%%%%%%%%%%%%%%%%%%%%%%%
\subsection{2-string tangles}

A (2-string) tangle is a pair $(B,T)$ where $B$ is a 3-ball and $T$ is a pair of disjoint arcs properly embedded in $B$. A tangle is untangled or rational if $T$ is isotopic (rel $\partial T$) to embedded arcs in $\partial B$. A tangle is prime if it is not rational and any 2-sphere intersecting $T$ transversely in two points bounts a ball with an unknotted spanning subarc of $T$. If we fix the boundary of every tangle to be a fixed set of points $\{NE,NW,SW,SE\}$, then we can `multiply' two tangles $(B,T_1)$ and $(B,T_2)$ by vertically stacking $T_1$ over $T_2$; we denote the new tangle by $(B,T_1*T_2)$. The numerator closure of a tangle, denoted by $N(T)$, is the link (or knot) obtained by connecting the north endpoints of $T$ (resp. south endpoints) with horizontal arcs. For a more in-depth introduction to the calculus of tangles, please see \cite{rational_tangles_Kauffman}.

\begin{lemma}[Lem 2 from \cite{LickorishTangles}]\label{lem:prime-tangles}
Let $(A, t)$ be a (two-string) tangle that is either prime or rational. Let $B$ be a 3-ball in the interior of $A$ which meets each component of $t$ in a single interval. Let $u$ be two arcs in $B$ such that $(B, u)$ is a prime tangle and $\partial u = B \cap  t $. Let $t'$ denote the two arcs $(t -(t \cap  B)) \cup u$; then $(A, t')$ is a prime tangle.
\end{lemma}

\begin{proposition}\label{ref_prop2}
Let $(B,T)$ be a rational tangle such that $T=T_1*T_2$,  then some $T_i$ is a rational tangle of the form $1/n$. 
\end{proposition}
%Comment: Creo que esto no se necesita: {\color{red} Suppose that the separating disk is incompressible in $B-\eta(T)$, }
%Comment: One could removeform the incompressibility condition in Proposition \ref{ref_prop2}. We only need it if it is incompressible so it is okay for us to leave it like that.
\begin{proof}
Observe that both $T_1$ and $T_2$ are locally unknotted, because $T$ is a rational tangle. So, for $i=1, 2$, $T_i$ is either a rational or prime tangle. If one of them were prime, let's say $T_1$, we can replace $T_1$ with a trivial tangle $S$ and obtain a tangle isotopic to $T_2$.  Using Lemma \ref{lem:prime-tangles} with $A' = B$, and $A =S * T_2 = T_2$ (which is either prime or rational), we conclude that $B$ must be prime, contradicting the fact that $B$ is rational. Hence, both tangles must be rational.

Now, by taking the numerator closure, we get that $N(T) = N(T_1) \# N(T_2)$. But we know that $N(T)$ is a prime knot \cite{tunnel_one_knots}, so either $N(T_1)$ or $N(T_2)$ is a trivial knot. But the only rational tangle $p/q$ with $N([p/q])$ a trivial knot is the tangle $\pm1/n$. 
%The numerator closure of $T$ is a 2-bridge knot which is prime. Thus, the numerator closure of some $T_i$ is an unknot. On the other hand, Lemma 2.3 of \cite{tsutsumi} implies that the separating disk cuts $B-\eta(T)$ into two genus two handlebodies. Hence, by Proposition \ref{ref_prop1}, the tangle $T_i$ is rational of Conway number $1/n$.
%TOTHINK: Otra opcion es usar un teorema 3.1 de https://www.mathi.uni-heidelberg.de/~banagl/pdfdocs/sumners/MPCPS90.pdf mas el hecho de que $N(T_i)$ es no nudo}
\end{proof}

\begin{proposition}\label{prop_tangle_argument}
Let $H$ be a genus two handlebody and $J=c_1\cup c_2\cup c_3$ a 1-submanifold separating $\partial H$ into two incompressible pairs of pants $S$ and $S_0$. If $c_1$ and $c_2$ are basic circles in $H$, then any incompressible pair of pants $K\subset H$ spanning $J$ is parallel to either $S$ or $S_0$. 
\end{proposition}
\begin{proof}
We think of $H$ as the complement of a trivial 2-string tangle $(B^3,T)$. One can see this by adding 2-handles to $H$ along $c_1$ and $c_2$: the resulting 3-manifold is a ball $B^3$ and the co-cores of the 2-handles are the strings of $T$. We can depict $(B^3,T)$ as in Figure \ref{fig:argumento_ovillo}, the meridional curve (slope $0$) corresponds to the curve  $c_3$ of $J$, and the meridians of each string are isotopic to $c_1$ and $c_2$. The top 3-holed disk in $\partial \left(B^3-\eta(T)\right)$ corresponds to $S_0$ and the bottom one to $S$. Furthermore, any other pair of pants $(K, \partial K) \subset (H, J)$ will correspond to a disk in $B^3$ with $\partial$-slope $0$ that intersects each string once. In particular, each $K$ will yield a factorization of $T$ as the product of two 2-string tangles $T=T_1* T_2$. By Proposition \ref{ref_prop2}, both are rational, and one, say $T_2$, is of the form $1/n$. This, in particular, gives us a parallelism between $K$ and $S$. Therefore, the pair $(H,J)$ is minimal.  
%FIXME: 
%
\begin{figure}
    \centering
    \includegraphics[width=9cm]{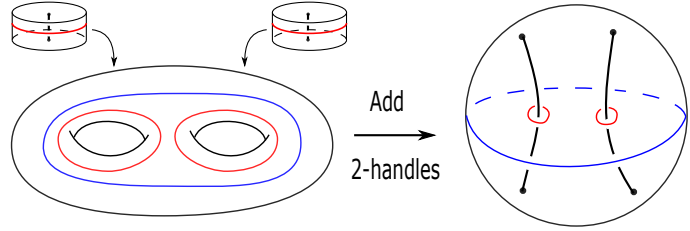}
    \caption{Adding two 2-handles along a pair of base curves creates a 2-string tangle }
    \label{fig:argumento_ovillo}
\end{figure}
\end{proof}

\begin{proposition}\label{prop_new_lem_3.5}
Let $H$ be a genus two handlebody and $c_1,c_2\subset \partial H$ a pair of disjoint essential simple closed curves. Let $M=H\underset{{c_1,c_2}}{\cup} T$ be the manifold obtained by gluing a solid torus $T$ along annular neighborhoods $A_1\cup A_2=\partial H \cap \partial T$ of $c_1\cup c_2$, such that each $A_i$ runs at least twice around $T$. 
If $M$ is a genus two handlebody, then $\{c_1,c_2\}$ are basis circles in $H$. 
\end{proposition}
%TODO: agregar prueba Jamboard 12 de Agosto 2022
%CAREFUL: Cambi\'e la notaci\'on del jamboard para que coincida con la del lemma 3.5 de luis
\begin{proof}
Let $B_1$ and $B_2$ be the complementary annuli of $A_1 \cup A_2$ in $\partial T$. Let $b_1$ and $b_2$ be the cores of $B_1$ and $B_2$, respectively. Observe that $b_1$ and $b_2$ lie on $\partial M$ and that they have companion annuli $A_1 \cup A_2 \cup B_2$ and $A_1 \cup A_2 \cup B_1$, respectively. 
By Proposition \ref{pro:gamma-has-companion-annulus-condition}, $b_1$ and $b_2$ are power circles in $\pi_1(M)$. In particular, there is a non-separating compressing disk $D$ for $\partial M$ such that $D \cap b_1 = \emptyset $. Take $C$ a parallel copy of $A_1$ in $M$, with $\partial C = b_1 \cup b_2$.  

Because any circle in $D$ is trivial in $\pi_1(M)$ and as the core of $C$ (parallel to $b_1$ and $b_2$) is a power circle in $M$, then $D\cap C$ can not have circles parallel to the core. That means that the circles (closed curves) in $D\cap C$ are trivial in $C$. By removing innermost circles at the time in $C$ of $C \cap D$, we can reduce the intersections between $D$ and $C$ to only arcs. 

Because $D$ is disjoint from one of the two boundaries of the annulus $C$, the intersection arcs of $C\cap D$ have both ends in the same components of $\partial C$. So, the intersection arcs of $D \cap C$ only have ended in one component of $\partial C$ (on $b_2$), so they are parallel to $b_2$. Then, removing one innermost arc at a time can eliminate all the intersections between $C$ and $D$. In conclusion, we can make $D$ and $C$ disjoint. In particular, we have made $D$ disjoint from $b_1$ and $b_2$.

Notice that we made $D$ disjoint from $C$. As $C$ is parallel to $A_1$, then we can make $D$ also disjoint from $A_1$. Similarly, we can make $A_2$ and $D$ disjoint. 

As all the modifications we perform on $D$ are isotopies, $D$ is still non-separating. Then $M' = \overline{M - \eta(D)}$ is a solid torus with $A_1$ and $A_2$ embedded in it. Recall that $A_1$ and $A_2$ separate a torus $T$ from $M$, hence the disk $D$ must lie either in $T$ or in the complement $H$. It is not $T$ because $A_1$ and $A_2$ go around the longitude at least twice, so any properly embedded in $T$ disjoint from $A_i$ ($i=1,2$) is parallel to $B_1 \cup B_2 = T - A_1 \cup A_2$,  hence parallel to $\partial M$. Therefore, $D$ lies on $H$ and not on $T$. In conclusion, $D$ is a compressing disk for $H$ and disjoint from $A_1$ and $A_2$. 

Then $H' =\overline{H - \eta(D)}$ is one or two solid tori depending on whether $D$ is separating. We can discard the former 
because the result of gluing two solid tori along two disjoint annuli on their boundary yields an oriented manifold with disconnected boundary or a non-orientable one with connected boundary. 

Let $T'_1$ and $T'_2$ be the two solid tori components to which $H'$ decompose, being $T'_i$ the one containing the annulus $A_i$ for $i=1,2$. This implies that if we cut $M'$ along $A_1 \cup A_2$, we obtain three solid tori; $T$, $T'_1$, and $T'_2$ with $A_i \subset \partial T'_i$ for $i=1, 2$.
We will conclude the proof by showing that $A_i$ loops around the longitude of $T_i$ just once for $i=1, 2$.

Recall that $M'$ is also solid tori, with two incompressible annuli $A_1$ and $A_2$ properly embedded. So, $A_1$ and $A_2$ are $\partial$-parallel to $\partial M'$. Hence, $A_i$ splits $M'$ into two solid tori; one of which $A_i$ goes along the longitude only once; we will show that that torus is $T'_i$. 
Cut $M'$ along $A_1$, we obtain two pieces; one of them is a product region i.e. $A_1 \times I$. Now, $A_2$ can not be inside that region, or it will imply that $A_2$ is parallel to $A_1$, which translates that the bounded region $T$ between $A_1$ and $A_2$ inside $M'$ is product; contradicting the fact that $A_i$ goes along the longitude at least twice. So, the parallel region for $A_1$ is $T'_1$. Similarly, we have that the parallel region cut out by $A_2$ is $T'_2$. This concludes the proof.
\end{proof}
%
%
%%%%%%%%%%%%%%%%%%%%%%%%%%%%%%%%%%%%%%%%%%%%%%%%%%%%%%%%%%%%%%%%%%%
\newpage
\section{Proof of Theorem \ref{thm:cota-2toro-en-dos-asas}}\label{section:cota-dos-asas}

We restate the theorem for the convenience of the reader. 
%DONE: Enunciar teorema (tsutsumi para toros 2-agujerados en cubos con dos asas)
\newtheorem*{thm:cota-asas}{Theorem \ref{thm:cota-2toro-en-dos-asas}}
\begin{thm:cota-asas}
Let V be a handlebody of genus two, and let $J = J_1 \cup J_2$ be an essential pair of simple closed curves that separates an annulus from $\partial V$. Then, $J$ bounds at most three mutually disjoint, non-parallel, incompressible, separating, twice-punctured tori in $V$.
\end{thm:cota-asas}

%DONE: Explicar que sin perdida de generalidad el complemento del anillo que bordea J debe ser incompresible. Llamar $F_0$ a tal complemento.
%DONE: Explicamos que la frontera de el cubo con asas - J es incompresible
%FIXME: Give the reference to why $H_1$ and $H_0$ are handlebodies.
Notice that any of those surfaces will cut out $V$ into a genus three handlebody $H_1$ and a genus two handlebody $H_0$. By taking the outermost surface, it will be enough to prove Theorem \ref{thm:cota-2toro-en-dos-asas} for $H_0$ instead of $V$. So, for now on and without lost of generality we will assume that $\partial V - J$ consists of an annulus $A$ and an \emph{incompressible} twice-punctured torus $F_0$.

%{\color{red} FYI. $F_1$ is now called $F$.}\\
%DONE: Explicar siempre hay una frontera compresión y que es hacia F_0. 
%FIXME: Probablemente necesita justificarse que $\partial D'$ es esencial en F
Let $F$ be an incompressible twice-punctured torus spanning $J$ and not parallel to $F_0$. Observe that $F$ has to be boundary compressible. By the following lemma, $F$ must be boundary compressible towards $F_0$. 

\begin{lemma}\label{lem:not_towards_A}
$F$ is not $\partial$-compressible towards $A$. 
\end{lemma}
\begin{proof}
Let $D$ be a boundary compressing disk for $F$ in $V$ and suppose that $D$ intersects $A$. Boundary compressing $A$ along $D$ yields a properly embedded disk $A'$. Observe that the $\partial A'$ is a curve in $F$ cutting it into a pair of pants and a once-punctured torus. Hence, $\partial A'$ is non-trivial and $A$ is a compressing disk for $F$. This contradicts the incompressibility of $F$. 
\end{proof}

%DONE: Explicar la operaci'on de compresi'on a lo largo del disco D (i.e. cortar F a lo largo de beta y pegar dos copias de D).
Let $D$ be a boundary compression for $F$ inside $V$. Denote by $\beta = D\cap F$ and $\alpha = D\cap F_0$ the arcs in $\partial D$. Since both $F$ and $F_0$ are incompressible, such arcs are essential in $F$ and $F_0$. %We now describe what occurs when boundary compressing $F$ along $D$. 
Remove from $F$ the interior of a small neighborhood of $\beta$. We obtain a surface $F' = F - \eta(\beta)$ with two marked parallel copies of $\beta$ on its boundary. Now glue two copies of $D$ to $F'$ along the $\beta$ arcs . We denote by $F|_D$ the resulting surface (see Fig. \ref{fig:surgery-along-a-disk}).
%DONE: Explicar que F|_D  tiene frontera J|_alpha
The boundary compression along $D$ induces a surgery %similar lower dimensional boundary compression 
on $J$ along the arc $\alpha$ which can be described as follows. Take a closed regular neighborhood $\eta(\alpha) = \alpha \times [-1,1] \subset F_0$, we set $J|_\alpha = (J - \eta(\partial \alpha)) \cup \alpha \times \{-1\} \cup \alpha \times\{1\}$ as the boundary compression of $J$ along $\alpha$ (see Fig. \ref{fig:surgery-along-an-arc}). Notice that $\partial (F|_D) = J|_\alpha$.

%TODO: Hacer una figura mostrando c'omo se hace la frontera compresi'on en un disco.
\begin{figure}[h!]
\centering
\labellist \small\hair 2pt 
\pinlabel {$D$}  at 135 70
\pinlabel {$F$}  at 180 30  
\pinlabel {$F|_D$}  at 485 90 
\endlabellist  
\includegraphics[width=11cm]{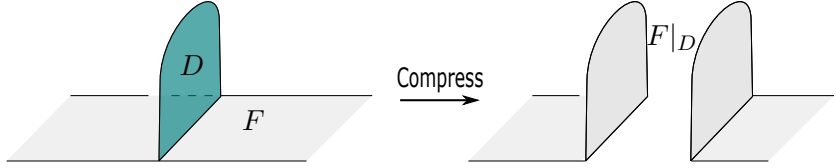}
\caption{At the right $F|_D$; $F$ after boundary compress along  $D$}
\label{fig:surgery-along-a-disk}
\end{figure}

%TODO: Hacer una figura mostrando c'omo se hace la cirug'a en un arco y como se hace en un disco.
\begin{figure}
    \centering
    \includegraphics[width=6cm]{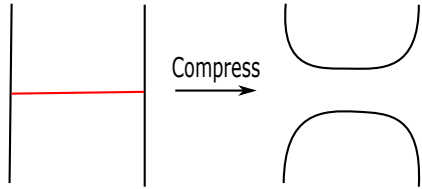}
    \caption{At the right $J|_\alpha$: $J$ after boundary compress along $\alpha$}
    \label{fig:surgery-along-an-arc}
\end{figure}

%DONE: Describir que después de cirugía aparecen 3 superficies. F'_0, F' y P.
After boundary compressing $F$ along $D$ we obtain a surface $F'=F|_D$ with boundary $J'=J|_\alpha$. The curves in $J'$ separate the boundary of $V$ in two surfaces, $F'_0$ and $P$. The surface $F'_0$ can be obtained from $F_0$ by removing an open regular neighborhood of $\alpha$ and $P$ is obtained from $A$ by adding a 2-dimensional 1-handle along $\eta(\alpha)$. 

%DONE: Demostrar que F'_0, F' y P son incompresibles.
\begin{lemma}\label{lem:incompressible_F'}
The surfaces $F'$, $F'_0$ and $P$ are incompressible in $V$.
\end{lemma}
\begin{proof}
Since $F'_0$ is contained in $F_0$ and the later is incompressible, then $F'_0$ is incompressible.  Similarly, $F'$ is incompressible because $F'$ retracts to $F-\eta(\beta)$ which is incompressible. 

%As for $P$, we have two cases, 
By construction, $P$ is either a once-punctured torus or a pair of pants depending on whether $\alpha$ connects the two boundaries of $A$ or not. 
If $P$ is a pair of pants, then any compressing disk $L$ for $P$ in $V$ will have $\partial L$ parallel to a component of $\partial P$, making $F'_0$ compressible. Hence, if $P$ is a pair of pants, then $P$ is incompressible. 
Now consider the case when $P$ is a once-punctured torus ans suppose that it has a compressing disk $D$. If $\partial D$ is a non-separating curve in $P$, then compressing $P$ along $D$ yields a disk with boundary $\partial P=J|_\alpha =\partial F'_0$; contradicting the fact that $F'_0$ is incompressible. A similar argument shows that $\partial D$ cannot be separating in $P$. Hence, such $D$ cannot exist, so $P$ is incompressible.
\end{proof}
%TODO: Decir que "Es suficiente probar una version de tsutsumi para superficies con caracteristica de euler -1"

Because of the way we constructed surfaces $F'$, $F'_0$, and $P$, they all have  Euler characteristics equal to $-1$. Let $K$ be another twice-punctured surface spanning $J$. Since, the surfaces $F_0,F$, and $K$ are `nested', the $\partial$-compressing disk $D$ must intersect $K$. In fact, $K$ intersects $D$ exactly in the middle (in an arc between $\alpha$ and $\beta$). So, we can boundary compress $K$ along the half disk of $D$ containing $\alpha$ and create an incompressible surface $K'$ as before: with $\chi(K')=-1$ and $\partial K' =  \partial P$.  
% Hence, if we manage to prove the following conjecture, Theorem \ref{thm:cota-2toro-en-dos-asas} will follow.
% \begin{conjecture}
% \label{conj:general-tsustumi}
% Let $V$ be a genus two handlebody, $J$ a separating 1-dimensional submanifold of $\partial V$ such that $\partial V - J$ consists of two connected surfaces each with Euler characteristic equal to $-1$. Then, there are at most four disjoint incompressible surfaces in $V$ spanning $J$ with $-1$ Euler characteristic and non-parallel between each other.
% \end{conjecture}
%
%
%TODO?: Quieren enunciar tsutsumi para superficies de caracteristica de euler -1 con tres fronteras? Lo podriamos poner aqui como un 'problema' o 'conjetura'.
%
%
%Instead of proving Conjecture \ref{conj:general-tsustumi}, w
% We will list %prove a much weaker proposition that is strong enough to cover  
% the possible surface arrangements that arise in our context (when $J$ has two components and $V-J$ has a pair of pants); and prove each case separately. We list the cases below.

%DONE: Describir los 3 tipos de arcos esenciales en un toro 2-agujerados: I, II y III
% Tipo I: Corta un toro agujerado
% Tipo II: Corta en un par de pantalones
% Tipo III: Corta en un toro agujerado y un anillo.
Modulo homeomorphism, there are three distinct essential, properly embedded arcs in a twice-punctured torus $T$. %As cutting by an arc increases the Euler characteristic by $1$, the resulting surface $T'$ satisfies that $\chi(T') = \chi(T)+1 = -1$. The number of boundary components of $T'$ can be either one (when the arc connects both boundary components of $T$) or three (when the arc connects one component). So, this limits the possibilities for $T'$ to the following three:
% \begin{enumerate}
% \item[I.] $T'$ is a once-puncture torus.
% \item[II.]  $T'$ is a pair of pants
% \item[III.]  $T'$ is a once-puncture torus and an annulus.
% \end{enumerate}
We will say that an arc is type I, II, or III if it cuts $T$ into a once-punctured torus, a pair of pants, or a disconnected surface, respectively. Figure \ref{fig:arcs-types} contains standard models for each type of arc. 
In general, the arcs $\alpha$ and $\beta$ can be of any type. The following lemmas reduce these possibilities by half (see Table \ref{tab:cases-alfa-beta-types}).

%FIXME: Improve figure
\begin{figure}
\centering
\labellist \small\hair 2pt 
\pinlabel {I}  at 280 230
\pinlabel {II}  at 40 230  
\pinlabel {III}  at 270 80 
\endlabellist     
\includegraphics[width=8cm]{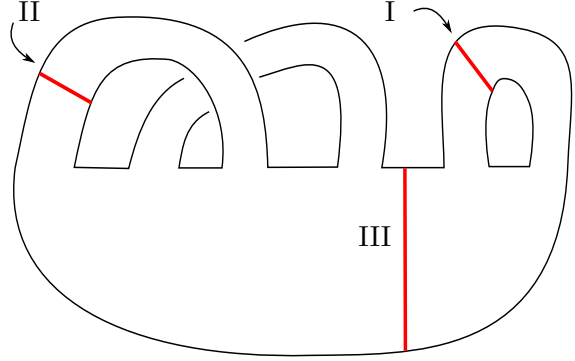}
\caption{Possibilities for an arc in a twice-punctured torus}
\label{fig:arcs-types}
\end{figure}

%DONE: Demostrar que Si \alpha es de tipo I si y sólo si \beta es de tipo I
\begin{lemma}
The arc $\alpha$ is type I if and only if $\beta$ is type I.
\label{lem:alphaI-iff-beta-I}
\end{lemma}
\begin{proof}
%This is an easy observation, because 
By construction, the ends of $\alpha$ are the same that those of $\beta$. So, $\alpha$ connects two different components of $J = \partial F_0 = \partial F$ if and only if $\beta$ does it. Since type I is the only arc type that connects two different components of $J$, the lemma holds.
\end{proof}

%DONE: Si \beta es de tipo III entonces \alpha es de tipo III.
\begin{lemma}
If $\beta$ is type III, then $\alpha$ is type III.
\label{lem:betaIII-implies-alpha-III}
\end{lemma}
\begin{proof}
If $\beta$ is type III then $F'$ is the disjoint union of an annulus $A_1$ and a once-punctured torus. Like any other incompressible surface in $V$, $A_1$ is $\partial$-compressible in $V$. %As $A_1$ is boundary compressible towards either $F'_0$ or  $P$ we claim that this implies that $A_1$ is boundary parallel. We prove this claim in the following paragraph.
Via a $\partial$-compression for $A_1$, we create a disk $A'_1\subset V$ with boundary in either $F'_0$ or $P$. Since both surfaces are incompressible, this means that $\partial(A'_1)$ bounds a disk $B$ in $\partial V$ disjoint from $J' = J|\alpha$. Because $V$ is irreducible, $A'_1$ and $B$ cobound a 3-ball, \emph{i.e} they are parallel in $V$. Hence $A_1$ is boundary parallel to either $F'_0$ or $P$. But $\partial A_1 \subset J' = \partial F'_0 = \partial P$, so one of $F'_0$ or $P$ must contain an annulus component. As $P$ is always connected, we conclude that $F'_0$ is annular and $\alpha$ is type III.
\end{proof}

%TODO: Escribir la tabla de casos (6 de Mayo 2022)
% alfa y beta sólo pueden ser I-I, II-II, III-III, III-II

%By Lemmas \ref{lem:alphaI-iff-beta-I} and \ref{lem:betaIII-implies-alpha-III} we can reduce the possibilities for the arcs $\alpha$ and $\beta$ to the four cases shown on Table \ref{tab:cases-alfa-beta-types}.
\begin{table}[h]
    \centering
        \begin{tabular}{c|c|c}
        Case & $\alpha$'s type  & $\beta$'s type\\
        \hline
        1 & I &I \\
        2 & II & II \\
        3 & III & III \\
        4 & III & II \\
        \end{tabular}
    \caption{The four possible cases for $\alpha$ and $\beta$}
    \label{tab:cases-alfa-beta-types}
\end{table}

We are left with finitely many possibilities for the types of $\alpha$ and $\beta$. We now proceed to work on each case separately.

\subsection*{Case 1: arc types I-I} 
%TODO: La prueba se reduce a Tsu-tsumi
In this case $P$, $F'$ and $F'_0$ are all once-punctured torus.
Any twice-punctured surface spanning $J$ between $F$ and $F_0$ will turn into a once-punctured torus $K'$ between $F'$ and $F'_0$. 
By Lemma \ref{lem:tsustumi}, we know that there is at most one such torus $K'$. 
Thus, there are at most three non-parallel twice-punctured tori: $F_0$, $F$, and $K$. %Because, by boundary compressing $K$, we can build a once-punctured torus $K'$ between $F'$ and $F'_0$. 
This concludes Case 1.

\begin{lemma}[Y. Tsutsumi \cite{tsutsumi}] \label{lem:tsustumi}
Let $V$ be a handlebody of genus two, and let J be an essential simple closed curve which separated $\partial V$. Then, $J$ bounds at most four disjoint, non-parallel, genus one incompressible surfaces in $V$.
\end{lemma}

\subsection*{Case 2: arc types II-II} 
%This case requires a lot more work than Case 1. 
In this case, $P$, $F'_0$, and $F'$ are all homeomorphic to a pair of pants. So it will be any other surface $K'$ resulting by boundary compressing a twice-punctured torus $K$ between $F_0$ and $F$ %. This is because $K$ has to intersect the $\partial$-compressing disk in a type II arc as well 
(see Lemmas \ref{lem:alphaI-iff-beta-I} and \ref{lem:betaIII-implies-alpha-III}). 
As before, it is enough to prove that there is at most one more surface $K'$ in $V$ spanning $J' = J|_\alpha = \partial P = \partial F'_0 = F'$. 
This is done in Lemma \ref{lem:tsutusmi-pair-of-pants} 
which is a `pair of pants version' of Lemma \ref{lem:tsustumi}. The proof of Lemma \ref{lem:tsutusmi-pair-of-pants} will be discussed in Section \ref{sec:proof-tsutsumi-pair-of-pants}.

\begin{lemma}\label{lem:tsutusmi-pair-of-pants}
Let $V$ be a handlebody of genus two, and let $J$ be a 1-submanifold separating $\partial V$ into two pairs of pants. Then, $J$ bounds at most four pairwise disjoint, non-parallel, incompressible pairs of pants in $V$.
\end{lemma}

\subsection*{Cases 3 and 4: arc types III-III and III-II} 
In this case, $P$ is a pair of pants and $F'_0$ is the disjoint union of an annulus and a once-punctured torus. %The surface $F'$ is homeomorphic to either $P$ or $F'_0$, depending whether the arc $\beta$ is of type II or III, respectively. 
Denote by $F_0, F_1, F_2, \dots, F_n=F\subset H$ the non-isotopic incompressible twice-punctured tori spanning $J$ between $F_0$ and $F$. Label them in such a way that, for each $0<i<n$, $F_i$ is inside the genus three handlebody bounded by $F_{i-1}\cup F_{i+1}$ (see Figure \ref{fig:casoIII_X}(a)). The goal is to show that $n\leq 2$. 

%DONE: Notation for P and (T\cup A)
As explained before, the $\partial$-compressing disk $D$ for $F=F_n$ can be chosen to intersect each $F_i$ in one essential arc $\beta_i$ satisfying $\partial \beta_i = \partial \alpha$. 
This way, after $\partial$-compressing each surface each $F_i$, we obtain a pair of pants $P_i$ or the disjoint union of an annulus $A_i$ and a once-punctured torus $T_i$, depending whether the arc $\beta_i$ has label II or III. Denote the curves in $J'$ by $J'_1\cup J'_2\cup \gamma$ where $J'_1=J_1$ and $J'_2\cup \gamma$ is the result of performing surgery to $J_2$ along $\alpha$. Label $\gamma$ so that $\partial T_i=\gamma$ and $\partial A_i = J'_1\cup J'_2$. 
By Lemma \ref{lem:betaIII-implies-alpha-III}, if $\beta_{i+1}$ is labeled III then $\beta_i$ is also labeled III. Thus there exists $1\leq k\leq n+1$ such that, after $\partial$-compression along $D$, we obtain $(n+2)$ incompressible surfaces $(T_0\cup A_0), (T_1\cup A_1), \dots, (T_{k-1}\cup A_{k-1}), P_k, \dots, P_n$, and $P_{n+1}=P$ embedded in $H$ and organized as in Figure \ref{fig:casoIII_X}. 

\begin{figure}
\centering
\labellist \small\hair 2pt 

\pinlabel{(a)} at 10 175

\pinlabel {$F_k$}  at 133 72
\pinlabel {$F_n$}  at 142 113
%\pinlabel {$F_{k-1}=F$}  at 167 40
\pinlabel {$F_{k-1}$}  at 105 35
\pinlabel {$F_1$}  at 70 10
\pinlabel {$F_0$}  at 20 140

\pinlabel{$J_1$} at 115 170
\pinlabel{$A$} at 160 170
\pinlabel{$J_2$} at 225 170
\pinlabel{$\beta_0$} at 290 170
\pinlabel{$F_0$} at 320 140
\pinlabel{{\color{red}$D$}} at 270 140

\pinlabel{$\beta_1$} at 285 105
\pinlabel{$\beta_{k-1}$} at 253 85
\pinlabel{$\beta_{k}$} at 210 85
\pinlabel{$\beta_{n}$} at 182 138

\pinlabel{(b)} at 430 175

\pinlabel{$J'_1,J'_2$} at 510 170
\pinlabel{$P_{n+1}$} at 570 170
\pinlabel{$\gamma$} at 630 170

\pinlabel{$P_n$} at 570 110
\pinlabel{$P_k$} at 570 70

\pinlabel {$A_{k+1}$}  at 465 60
\pinlabel {$A_{1}$}  at 420 75
\pinlabel {$A_{0}$}  at 425 140
\pinlabel {$T_{k+1}$}  at 640 45
\pinlabel {$T_{1}$}  at 725 65
\pinlabel {$T_{0}$}  at 720 140

\endlabellist  
\includegraphics[width=14cm]{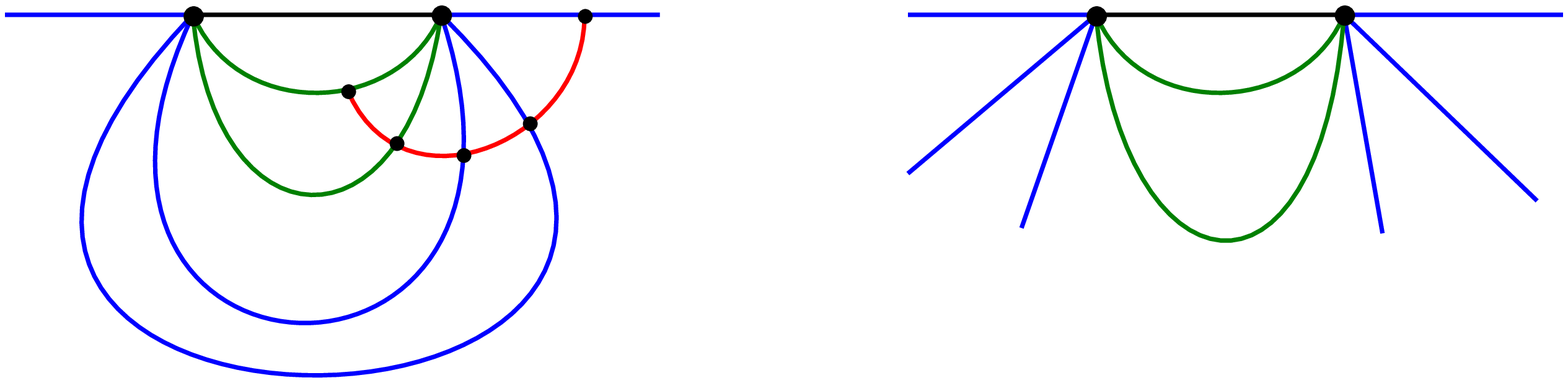}
\caption{Case III-X. (a) Before and (b) after $\partial$-compressing along $D$.}
\label{fig:casoIII_X}
\end{figure}

%DONE: Explain why the annuli are isotopic rel boundary.
We first focus on the incompressible annuli $\{A_i:0\leq i \leq k-1\}$ embedded in $H$. Suppose that $A_i$ is not $\partial$-parallel for some $i\neq 0$. Then, $A_i$ is a companion annulus for $J'_1=J_1$ in $H$. By Propositions \ref{pro:gamma-is-not-primitive-or-power} and \ref{pro:gamma-has-companion-annulus-condition}, %By Lemma 3.3 of \cite{luis1}, 
$J_1$ is a power circle in $H$ and $\partial H-\eta(J_1)$ is compressible. The latter contradicts the original assumption that $F_0=\partial H- \eta(J_1)$ is incompressible. Hence, all the annuli $A_i$ are isotopic to $A_0$. 
For each $k\leq i\leq n+1$, define $T_i\subset H$ to be the once-punctured torus with $\partial T_i=\gamma$ obtained by pushing $P_i\cup A_0$ towards the interior of $H$. %This procedure gives us $(n+2)$ tori inside $H$ spanning $\gamma$. If we can tell some of them apart, we would obtain an upper bound for $n$ using Lemma \ref{lem:tsustumi}. 

%TODO Paso2. Decir que los toros 1-agujerados son no paralelos salvo el ultimo pantalon ($P_l\cup A$) con el ultimo toro ($F_1$). 
\begin{lemma}\label{lem:paralel_tori}
Let $\mathcal{T}=\{T_i:0\leq i\leq n+1\}$ be the collection of $(n+2)$ tori defined as above. 
Then $\mathcal{T}$ contains only incompressible tori in $H$ and for every pair $(i,j)\neq (k-1,k)$, $T_i$ is not isotopic to $T_j$.
\end{lemma}
\begin{proof}
One can use the proof of Lemma \ref{lem:incompressible_F'} to check that each torus is incompressible. 
Suppose that $T_i$ is isotopic to $T_j$ for some $i<j$. It is not hard to see that $T_l$ is isotopic to both $T_i$ and $T_j$ for $i<l<j$. Thus, it is enough to check the result for $(i,i+1)\neq (k-1,k)$. 

%DONE Paso2(a). $F_j$ no es paralelo a $F_{j+1}$. $j=1\dots m$
Let $0\leq i<k-1$. By construction, $F_i$ (resp. $F_{i+1}$) is obtained by adding a 2-dimensional 1-handle connecting $A_i$ and $T_i$ (resp. $A_{i+1}$ and $T_{i+1}$). Here, the cocore of the handle corresponds to the arc $\beta_i$ (resp. $\beta_{i+1}$). Thus, since $A_i$ and  $A_{i+1}$ are isotopic, a parallelism between $T_i$ and $T_{i+1}$ can be extended to a parallelism between $F_i$ and $F_{i+1}$. In particular, $T_i$ and $T_{i+1}$ cannot be isotopic. 

%TODO Paso2(b). $P_i\cup A$ no es paralelo a $P_{i+1}\cup A$. $i=1\dots n$
Let $k\leq i\leq n$ and suppose that $T_i$ is parallel to $T_{i+1}$. We claim that $P_i$ and $P_{i+1}$ must be parallel. 
The region $R$ between these tori is homeomorphic to a product $T\times [0,1]$ where $\gamma =\partial T \times \{1/2\}$ and $T_{i}$ (resp. $T_{i+1}$) can be identified with $T\times \{0\}$ (resp. $T\times \{1\}$). Let $\rho_0$ (resp. $\rho_1$) be the core of the annulus $A_0$ inside $T_i=T\times \{0\}$ (resp. $T_{i+1}=T\times \{1\}$). By construction, there is an embedded annulus $S\subset R$ with $\partial S=\rho_0\cup \rho_1$. 
By Lemma 2.1 of \cite{TaoLi_saddles}\footnote{Although this lemma is proven for closed surfaces, the proof holds for Heegaard surfaces with boundary.}, we can isotope the interior of $S$ so that it intersects all but finitely many surfaces $T\times \{t\}$ transversely in simple closed curves, and the non-transverse intersections correspond to saddles or circle tangencies of $S$ (see Fig. 2.1(a) of \cite{TaoLi_saddles}). 
Saddle tangencies change the Euler characteristic by $-1$ and $\chi(S)=0$. Thus $S$ only has circle tangencies. In that case, it is not hard to see that $S$ can be further isotoped to intersect each torus $T\times \{t\}$ in one circle; i.e., $S$ is a vertical annulus. Thus, the trivial foliation of $R$ restricts to a parallelism between $P_{i}$ and $P_{i+1}$.

Let $M$ be the product region between $P_{i}$ and $P_{i+1}$, % and let $M'$ be the region between $F_n$ and the original annulus $A$. 
we think of $J'$ as a suture for $\partial M$. If $k\leq i <n$, `undoing' the $\partial$-compression along $D$ corresponds to attaching a 3-dimensional 1-handle to $M$ along a neighborhood of $J'$. Hence, $F_i$ and $F_{i+1}$ are also parallel which is impossible. If $i=n$, since $P_{n+1}=P$ is a subset of $\partial H$, `undoing' the $\partial$-compression does not change $M$. This only changes the suture $J'\subset \partial M$ via surgery along an arc $x\subset P_{n+1}$ connecting $J'_2$ to $\gamma$. Since $M$ is a product, we can find a compressing disk for $M$ intersecting $P_n$ in an arc disjoint from $x$. This disk becomes a $\partial$-compressing disk for $F_n$ in $M$, contradicting the statement of Lemma \ref{lem:not_towards_A}. Hence, the regions $T_i$ and $T_{i+1}$ are not isotopic. 
\end{proof}

%TODO Paso3. Usar tsutsumi para decir que $m+n \leq 5$. Más aun,  si $m+n=5$ entonces $F_1$ es paralela $P_n \cup A$
By Lemma \ref{lem:paralel_tori}, $\mathcal{T}$ has least $n+1$ non-isotopic incompressible tori in $H$ with boundary $\gamma\subset \partial H$. Thus, by Lemma \ref{lem:tsustumi} we obtain that $n+1 \leq |\mathcal{T}|\leq 4$. In particular, $n\leq 3$.
%In order to achieve our goal ($n\leq 2$), it is enough to rule out the case $n=3$. 

Suppose $n=3$. Here, the surfaces $T_{k-1}$ and $T_k$ must be isotopic. Denote by $H_4$ the region in $H$ between $T_4$ and $T_5=P\cup A_0$. By construction, $J_1\subset \partial A_0$ is a non-separating curve in $\partial H_4$. 
Observe that the region $R\subset H$ bounded by the original annulus $A$ and the twice-punctured torus $F_4$ is isotopic to $H_4$ in $H$. One can see this by noting that $\partial$-compressing along $D$ changes such region by adding a 3-ball along one 2-disk on its boundary. 
Suppose that $J_1$ is a primitive or power circle in $H_4$, then so it is in $R$. Then Proposition \ref{pro:gamma-is-not-primitive-or-power} %Lemma 3.3 of \cite{luis1} 
would imply that $F_4=\partial R-\eta(J_1)$ is compressible in $H_4$ and so in $H$, which is a contradiction. 
Hence, $J_1$ cannot be a primitive or power circle in $H_4$.

Suppose first that $k=4$, then $T_4$ and $T_5=P\cup A_0$ cobound a product region $H_4\cong T\times I$. By construction, the core of $A_0$ is a non-separating simple closed curve in $T$, so it is primitive in $H_4$. Since, $J_1\subset \partial A_0$, we conclude that $J_1$ is primitive in $H_4$. This is impossible by the previous paragraph. 
% On the other hand, observe that region $R\subset H$ bounded by the original annulus $A$ and the twice-punctured torus $F_4$ is isotopic to $H_4$ in $H$. One can see this by noting that $\partial$-compressing along $D$ changes such region by adding a 3-ball along one 2-disk on its boundary. 
% Hence, $J_1$ is a primitive circle in $R$. Lemma 3.3 of \cite{luis1} implies that $F_4=\partial R-\eta(J_1)$ is compressible in $H_4$ and so in $H$, a contradiction. 

Suppose that $1\leq k<4$. By Lemma \ref{lem:paralel_tori}, $T_5=A_0 \cup P$ and $T_4=A_0\cup P_4$ are non-isotopic tori. Since $|\mathcal{T}|=4$, Proposition \ref{pro:lem_3.10} %Remark 3.8 and Lemma 3.10 of \cite{luis1} 
implies that the region cobounded by $T_4$ and $T_5$ is of type $(1,p)$ for some $p>1$. 
The core of $A_0$ in $T_4$ and the core of $A_0$ in $T_5$ are coannular in $H_4$. Lemma \ref{lem:simple_pair_annulus} implies that the core of $A_0$ in $T_5$ is a power circle in $H_4$. In particular, $J_1 \subset \partial A_0$ is also a power (or primitive) circle in $H_4$ which yields a contradiction. 

Therefore, $n$ must be at most 2. This concludes this case and the proof of Theorem \ref{thm:cota-2toro-en-dos-asas}.

% Paso1. Por lemma \ref{lem:betaIII-implies-alpha-III}, las superficies est\'an en orden. $P_1, P_2,\dots, P_n, F_m, \dots, F_1$ y $\partial H=F_m\cup A\cup P_1$. %TODO anadir imagen
%
% Paso1.5. Todos los anillos son paralelos
%
% Paso2. Decir que los toros 1-agujerados son no paralelos salvo el ultimo pantalon ($P_l\cup A$) con el ultimo toro ($F_1$). \\
% Paso2(a). $P_i\cup A$ no es paralelo a $P_{i+1}\cup A$. $i=1\dots n$\\
% Paso2(b). $F_j$ no es paralelo a $F_{j+1}$. $j=1\dots m$\\
% Paso2(c) Podria ser que $F_1$ sea paralela a $P_n \cup A$ 
%
%
% Paso3. Usar tsutsumi para decir que hay a lo mas cinco superficies con $\chi(K)=-1$ spanning $J'$. Es decir $m+n \leq 5$, más aun,  si $m+n=5$ entonces $F_1$ es paralela $P_n \cup A$
%
%
% Paso4. Nuevo Statement del lemma 12 ``Hay a lo mas cuatro" \\
% Dividir en casos, para las parejas $(m,n) = (1,4), (2,3), (3,2), (4,1)$\\
% Paso 4 (a) Caso $(m,n) = (1,4)$. La $J$ original es primitiva en el cubo con asas entre $P_1\cup A$ y $F_1$. De aqui, $\hat{F}_1$ (el toro 2-agujerado) es compresible (correo de Ago-26 + sept-6)\\
% Paso 4 (b) Caso $(m,n) = (2,3), (3,2), (4,1)$. La $J$ original es potencia en el cubo entre $P_1\cup A$ y $P_2\cup A$. De aqui, $\hat{P}_2$ (el toro 2-agujerado) es compresible. (sept-6)
%
% \begin{lemma}
% Sea $(H,\gamma)$ de tipo $(1,p)$ con $\partial H = F_1 \cup_\gamma F_2$, $F_i$ es un toro 1-agujerado. Sea $\beta_i\subset F_i$ es una curva cerrada simple no separante. Si $\beta_1$ y $\beta_2$ bordean un anillo en $H$, entonces $\beta_i$ es potencia en $H$. 
% \end{lemma}

\section{Proof of Lemma \ref{lem:tsutusmi-pair-of-pants}}\label{sec:proof-tsutsumi-pair-of-pants}
%NOTE: Las notas están en (13 de Junio del 2022)

%DONE: Introducir H_0 y H_1. explicar que es suficiente probar que algún par (H_i,J) es mínimo.

Denote by $S_0$ and $S_1$ the two pairs of pants in $\partial V-\eta(J)$. Let $S$ be a pair of pants spanning $J$ not parallel to either $S_0$ or $S_1$. Let $H_0$ and $H_1$ be the result of splitting $V$ along $S$, being $H_i$ the one containing $S_i$.
Because $S$ is incompressible, by Lemma 2.3 of \cite{tsutsumi}, $H_0$ and $H_1$ are genus two handlebodies. Moreover, $J$ separates $\partial H_i$ into two pairs of pants; $S$ and $S_i$. 

We will show that for some $i \in \{0,1\}$, the only incompressible pairs of pants bounded by $J$ in $H_i$ are $S$ and $S_i$. Using the notation in \cite{luis1}, such a pair $(H_i, J)$ is called a \emph{minimal pair}. 
Lemma \ref{lem:tsutusmi-pair-of-pants} follows from this claim. By way of contradiction, suppose that there are five non-isotopic pairs of pants in $V$. The one in the middle will split $V$ into two non-minimal pairs, contradicting the minimality of either side. %So, we are going to focus on proving that one of the pairs $(H_i, J)$ is minimal.

%DONE: Usar que existe otro disco de frontera compresión para dividir en casos:
Observe that $S$ is $\partial$-compressible in $V$. Without loss of generality, suppose that $S$ is $\partial$-compressible towards $S_0$. Let $D\subset H_0$ be a $\partial$-compressing disk for $S_0$. By definition, $\partial D=\alpha \cup \beta$ where $\alpha\subset S_0$ and $\beta\subset S$ are properly embedded arcs. Since $S_0$ and $S$ are incompressible, the arcs $\alpha$ and $\beta$ are essential in $S_0$ and $S$, respectively. 
Observe that if $D$ separates $H_0$, then $\partial D$ separates $\partial H_0$. In particular, both $\alpha$ and $\beta$ separate $S_0$ and $S$, respectively. Also, recall that whether $\alpha$ separates $S_0$ or not is determined by $\partial \alpha = \partial \beta$. Thus, we have three possibilities depending on whether $D$ and $\alpha$ separate. 
% Caso 1: Disco separante, alfa y beta separante (son otros alfas y betas)
% Caso 2: Disco no-separante, alfa y beta separantes
% Caso 3: Disco y arcos no-separantes.

\subsection*{Case (i): $D$, $\alpha$ and $\beta$ are non-separating}
%DONE: Probar caso 3. Es fácil pues H_0 comprimido es un toro con dos anillos compañeros.
Denote by $H'_0$ the 3-manifold obtained by compressing $H_0$ along $D$. Since $D$ is non-separating, $H'_0$ is a solid torus. 
We can extend the surfaces $S_0$ and $S$ in the boundary of $H'_0$ to obtain two annuli $S'_0$ and $S'$ (see Figure \ref{fig:extending_suture}). Observe that the incompressibility of $S_0$ and $S$ implies that $S'_0$ and $S'$ are incompressible annuli in $H'_0$. In particular, any other annuli $K'\subset H'_0$ with $|K'|=|S'|$ and $\partial K'=\partial S'$ will be boundary parallel to $S'$ or $S'_0$. Hence, $H_0$ is minimal since any other incompressible pair of pants $K\subset H_0$ spanning $J$ will yield other annuli $K'$ in $H'_0$. 

\begin{figure} 
    \centering
    \includegraphics{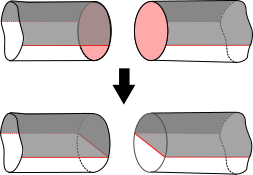}
    \caption{How to extend the surfaces $S_0$ and $S$ after compressing $H_0$.}
    \label{fig:extending_suture}
\end{figure}

\subsection*{Case (ii): $D$, $\alpha$ and $\beta$ separate}
%DONE: Probar caso 1. Aparecen curvas c1, c2 potencia en H0 que deberán ser primitivas en H1.
In this case, the result of compressing $H_0$ along $D$ is the disjoint union of two solid tori $V_1 \cup V_2$. Since $\alpha$ is separating in $S_0$, it connects the same component of $J=\partial S_0$. 
Denote the other two curves in $J$ by $c_1$ and $c_2$ where $c_i\subset V_i$.  Since $J$ is non-trivial in $H_0$, $c_i$ is an essential curve in $\partial V_i$ representing a non-trivial loop in $V_i$. Observe that the fundamental group of $H_0$ splits as the free product of the fundamental groups of $V_1$ and $V_2$. Thus, $c_i$ is a power circle (resp. primitive) in $H_0$ if and only if $c_i$ runs at least twice around $V_i$ (resp. is homotopic to a core of $V_i$).

We now discuss what occurs with the surfaces in $H_0$ after compressing along $D$. In this case, the surface $S_0$ (resp. $S$) extends to two annuli $S^{1}_{0}\cup S^{2}_{0}$ (resp. $S^{1}\cup S^{2}$) where $S^{i}_0\subset \partial V_i$ (resp. $S^i\subset \partial V_i$).  
Any other pair of pants $K\subset H_0$ spanning $J$ yields two disjoint annuli $K'_1\cup K'_2$ where $K'_i\subset \partial V_i$. The boundaries of $S^i_{0}$, $S^i$, and $K'_i$ are isotopic to $c_i$ in $\partial V_i$. 
In particular, if $c_i$ is a primitive circle in $V_i$, then $S^i_{0}$, $S^i$, and $K'_i$ are all isotopic in $V_i$ relative its boundary. If $c_i$ is a power circle in $V_i$, then $K'_i$ is isotopic to either $S^i_{0}$ or $S^i$ relative its boundary. Hence, if at least one of the $c_i$ curves is primitive, then any pair of pants $K$ will be isotopic to either $S_0$ or $S$. Thus, $(H_0,J)$ will be a minimal pair. 

% DONE: Demostrar que H1 más dos 2-asas en c1 y c2 es un tangle racional.
To end this subcase, it is enough to assume that both $c_1$ and $c_2$ are power circles in $H_0$. By the above, $c_1, c_2\subset \partial H_1$ satisfy the conditions of P oposition \ref{pro:lem_3.5_part2}. %Part 2 of Lemma 3.5 of \cite{luis1}. 
So $c_1$ and $c_2$ are basic circles in $H_1$. %By the 2-handle addition Theorem \cite{Casson_Gordon_red, jaco_2h} applied to $c_2\subset \partial H_1 - c_1$, there is a properly embedded disk $D'\subset H_1$ separating $c_1$ and $c_2$. % {\color{red} [I took this from defn of basic pair in \cite{luis1}]}.  
%Observe that if $D'$ were a $\partial$-compressing disk for $S$ in $H_1$, we could reverse the roles of $H_0$ and $H_1$ and repeat the argument in this subcase. Given that we have no control in the intersection $\partial D'\cap J$, we will proceed in a different way. 
%
% DONE: Insertar argumento del tangle para probar que (H1,J) es minimo
%UPDATE: el argumento del tangle lo hicimos su propia proposicion en los preliminares
Proposition \ref{prop_tangle_argument} applied to $H=H_1$ concludes that $(H_1,J)$ is a minimal pair, as desired. 
% We think of $H_1$ as the complement of a trivial 2-string tangle $(B^3,T)$. One can see this by adding 2-handles to $H_1$ along $c_1$ and $c_2$: the resulting 3-manifold is a ball $B^3$ and the co-cores of the 2-handles are the strings of $T$. We can depict $(B^3,T)$ as in Figure \ref{fig:argumento_ovillo}, where the meridional curve (slope $0$) corresponds to the third curve $c_3$ in $J$ and the meridians of each string are curves isotopic to $c_1$ and $c_2$. The top 3-holed disk in $\partial \left(B^3-\eta(T)\right)$ corresponds to $S_0$ and the bottom one to $S$. Furthermore, any other pair of pants $K\subset H_1$ will correspond to a disk in $B^3$ with slope $0$ intersecting each string once. In particular, each $K$ will yield a factorization of $T$ as the product of two 2-string tangles $T=T_1* T_2$. By Proposition \ref{ref_prop2}, both are rational and one, say $T_2$, is of the form $1/n$. This, in particular, gives us a parallelism between $K$ and $S$. Therefore, the pair $(H_1,J)$ is minimal. 
% %
% \begin{figure}
%     \centering
%     \includegraphics[width=9cm]{argumento_ovillo.PNG}
%     \caption{add caption}
%     \label{fig:argumento_ovillo}
% \end{figure}

\subsection*{Case (iii): $D$ does not separate but $\alpha$ and $\beta$ separate}
%DONE: Probar caso 2. (13 de Junio)
%DONE: Paso 1. Es suficiente probar el caso cuando c1 y c2 son potencia. Si no, H0 es minimo. 
In this case, the result of compressing $H_0$ along $D$ is a solid torus $H'_0$. Denote the curves in $J$ by $c_1 \cup c_2 \cup c_3$. Since $\alpha$ separates $S_0$, it connects the same component of $J=\partial S_0$, namely $c_3$. 
After compression, $J|_D$ becomes four loops in the boundary torus $\partial H'_0$ with the same slope. Here, $c_1\cup c_2\subset J|_D$ and $c_3$ breaks into two circles. Since $S_0$ is incompressible, $c_1$ and $c_2$ do not bound disks in $H'_0$. So they are either primitive or power circles in $H'_0$. 

Suppose that $c_1$ is primitive in $H'_0$. 
First observe that the surfaces $S$ and $S_0$ become two annuli which union is $\partial H'_0$. Also, any pair of pants $K\subset H_0$ spanning $J$ will turn into two annuli properly embedded in $H'_0$ spanning $J|_D$. The fact that $c_1$ is primitive implies that $J|_D$ intersects each meridian disk in exactly four points, one per curve. In particular, $K|_D$ is forced to be isotopic to either $S|_D$ or $S_0|_D$. Hence, $(H_0,J)$ is minimal if $c_1$ is primitive in $H'_0$. The same is true for $c_2$. 

In remains to discuss the case when $c_1$ and $c_2$ are both power circles in $H'_0$. Let $H'_1=\overline{V-H'_0}$ be the complement of $H'_0$ in the genus two handlebody $V$. Observe that $H'_1$ is the result of attaching the ball $\eta(D)$ along $\eta(\alpha)$ so $H'_1$ is also a genus two handlebody. Furthermore, $H'_1\cap H'_0$ is the union of two annular neighborhoods of two curves isotopic to $c_1$ and $c_2$. By Proposition \ref{prop_new_lem_3.5} with $T=H'_0$, $H=H'_1$ and $M=V$, we conclude that $\{c_1, c_2\}$ are basic circles in $H'_1$. Since $H'_1=H_1\cup_{\eta(\alpha)} \eta(D)$, it follows that $\{c_1,c_2\}$ are also basic circles in $H_1$. Proposition \ref{prop_tangle_argument} applied to $H=H_1$ concludes that $(H_1,J)$ is a minimal pair. 
This finishes the proof of Lemma \ref{lem:tsutusmi-pair-of-pants}.

\bibliographystyle{alpha}
\bibliography{references} % see references.bib for bibliography management

$\quad$ \\
Rom\'an Aranda, Binghamton University, Vestal, NY, USA\\
email: \texttt{jaranda@binghamton.edu}
$\quad$ \\
Enrique Ram\'irez-Losada, Centro de Investigaci\'on en Matem\'aticas, Guanajuato, Gto, Mexico\\
email: \texttt{kikis@cimat.mx} 
$\quad$ \\
Jes\'us Rodr\'iguez-Viorato, Centro de Investigaci\'on en Matem\'aticas, Guanajuato, Gto, Mexico\\
email: \texttt{jesusr@cimat.mx}

\end{document}